\pdfoutput=1

\documentclass[a4paper,UKenglish]{amsart}

\usepackage[utf8]{inputenc}
\usepackage[T1]{fontenc}
\usepackage{lmodern}
\usepackage{amssymb}
\usepackage{mathtools}
\usepackage{tikz-cd}
\usepackage{babel}
\usepackage{microtype}
\usepackage{hyperref}

\DeclareFontFamily{OMX}{lmex}{}
\DeclareFontShape{OMX}{lmex}{m}{n}{<->lmex10}{}

\tikzcdset{arrow style=Latin Modern}

\theoremstyle{plain}
\newtheorem{prop}{Proposition}
\newtheorem{coro}[prop]{Corollary}
\newtheorem{lemm}[prop]{Lemma}

\theoremstyle{remark}
\newtheorem{exem}[prop]{Example}

\DeclareMathOperator{\Hom}{Hom}
\DeclareMathOperator{\End}{End}
\DeclareMathOperator{\Ext}{Ext}
\DeclareMathOperator{\Mod}{Mod}
\DeclareMathOperator{\Fil}{Fil}
\DeclareMathOperator{\Gr}{Gr}
\DeclareMathOperator{\Ind}{Ind}
\DeclareMathOperator{\cind}{c-ind}
\DeclareMathOperator{\St}{St}
\DeclareMathOperator{\Ord}{Ord}
\DeclareMathOperator{\Art}{Art}
\DeclareMathOperator{\Noe}{Noe}
\DeclareMathOperator{\Pro}{Pro}
\DeclareMathOperator{\Def}{Def}
\DeclareMathOperator{\card}{card}
\DeclareMathOperator{\car}{char}
\DeclareMathOperator{\e}{e}

\newcommand{\iso}{\xrightarrow{\sim}}
\newcommand{\vertiso}{\rotatebox{-90}{\(\sim\)}}
\newcommand{\middlevert}{\;\middle|\;}
\newcommand{\otimesh}{\mathbin{\widehat{\otimes}}}
\newcommand{\Set}{\mathrm{Set}}
\newcommand{\Grp}{\mathrm{Grp}}
\newcommand{\GL}{\mathrm{GL}}
\newcommand{\Hr}{\mathrm{H}}
\newcommand{\HOrd}[1][\bullet]{\Hr^{#1}\!\Ord}
\newcommand{\cont}{\mathrm{cont}}
\newcommand{\univ}{\mathrm{univ}}
\newcommand{\adm}{\mathrm{adm}}
\newcommand{\lfin}{\mathrm{l.fin}}
\newcommand{\ord}{\mathrm{ord}}
\newcommand{\der}{\mathrm{der}}
\newcommand{\ab}{\mathrm{ab}}
\newcommand{\smc}{\mathrm{sc}}
\newcommand{\N}{\mathbb{N}}
\newcommand{\Z}{\mathbb{Z}}
\newcommand{\Q}{\mathbb{Q}}
\newcommand{\F}{\mathbb{F}}
\newcommand{\Bc}{\mathcal{B}}
\newcommand{\Oc}{\mathcal{O}}
\newcommand{\mf}{\mathfrak{m}}
\newcommand{\Gb}{\mathbf{G}}
\newcommand{\Bb}{\mathbf{B}}
\newcommand{\Sb}{\mathbf{S}}
\newcommand{\Pb}{\mathbf{P}}
\newcommand{\Lb}{\mathbf{L}}
\newcommand{\Ub}{\mathbf{U}}
\newcommand{\Qb}{\mathbf{Q}}
\newcommand{\Mb}{\mathbf{M}}
\newcommand{\Nb}{\mathbf{N}}
\newcommand{\Zb}{\mathbf{Z}}
\newcommand{\Gbt}{\widetilde{\Gb}}
\newcommand{\Gt}{\widetilde{G}}
\newcommand{\GU}{\langle\prescript{G}{}{U}\rangle}

\hypersetup{pdfinfo={
	Title={Functorial properties of generalised Steinberg representations},
	Author={Julien Hauseux, Tobias Schmidt, Claus Sorensen},
	Subject={22E50},
	Keywords={local fields, reductive groups, generalised Steinberg representations, extensions, deformations}
}}

\title{Functorial properties of generalised Steinberg representations}

\author[J.~Hauseux]{Julien Hauseux}
\address{
Université de Lille\\
Département de Mathématiques\\
Cité scientifique, Bâtiment M2\\
59655 Villeneuve d'Ascq Cedex\\
France
}
\email{\href{mailto:julien.hauseux@math.univ-lille1.fr}{julien.hauseux@math.univ-lille1.fr}}
\thanks{J.H. was partly supported by EPSRC grant EP/L025302/1.}

\author[T.~Schmidt]{Tobias Schmidt}
\address{
Institut de Recherche Mathématique de Rennes\\
Campus Beaulieu\\
35042 Rennes Cedex\\
France 
}
\email{\href{mailto:tobias.schmidt@univ-rennes1.fr}{tobias.schmidt@univ-rennes1.fr}}

\author[C.~Sorensen]{Claus Sorensen}
\address{
Department of Mathematics, UCSD\\
9500 Gilman Dr.\@ \#0112\\
La Jolla, CA 92093-0112\\
USA
}
\email{\href{mailto:csorensen@ucsd.edu}{csorensen@ucsd.edu}}

\makeatletter
\def\@@and{\unskip}
\makeatother

\begin{document}

\begin{abstract}
Let $G$ be the $F$-points of a connected reductive group over a non-archimedean local field $F$ of residue characteristic $p$ and $R$ be a commutative ring.
Let $P=LU$ be a parabolic subgroup of $G$ and $Q$ be a parabolic subgroup of $G$ containing $P$.
We study the functor $\mathrm{St}_Q^G$ taking a smooth $R$-representation $\sigma$ of $L$ which extends to a representation $\mathrm{e}_G(\sigma)$ of $G$ trivial on $U$ to the smooth $R$-representation $\mathrm{e}_G(\sigma) \otimes_R \mathrm{St}_Q^G(R)$ of $G$ where $\mathrm{St}_Q^G(R)$ is the generalised Steinberg representation.
\end{abstract}

\maketitle

\section{Introduction}

Let $F$ be a non-archimedean local field of residue characteristic $p$.
Let $\Gb$ be a connected reductive algebraic $F$-group and $G$ denote the topological group $\Gb(F)$.
In a recent paper (\cite{AHHV}), Abe, Henniart, Herzig, and Vignéras have classified the irreducible admissible smooth representations of $G$ over an algebraically closed field of characteristic $p$ in terms of supersingular representations of Levi subgroups.
There are two functorial steps in the construction of the irreducible representations.
In this article, we study the behaviour of extensions and deformations under the first step (extension to a larger parabolic subgroup and twist by a generalised Steinberg representation).
For the second step (parabolic induction), this has been done in \cite{JHD,HSS} when $\car(F)=0$ and \cite{JHE} when $\car(F)=p$.

Let $R$ be a commutative ring.
We write $\Mod_G^\infty(R)$ for the category of smooth $R$-representations of $G$ (i.e.\@ $R[G]$-modules $\pi$ such that for all $v \in \pi$ the stabiliser of $v$ is open in $G$) and $R[G]$-linear maps.
It is an $R$-linear abelian category.
When $R$ is noetherian, we write $\Mod_G^\adm(R)$ for the full subcategory of $\Mod_G^\infty(R)$ consisting of admissible representations (i.e.\@ those representations $\pi$ such that $\pi^H$ is finitely generated over $R$ for any open subgroup $H$ of $G$).
It is closed under passing to subrepresentations and extensions, thus it is an $R$-linear exact subcategory, but quotients of admissible representations may not be admissible when $\car(F)=p$ (see \cite[Example 4.4]{AHV}).

We fix a parabolic subgroup $\Pb=\Lb\Ub$ of $\Gb$ and we let $\GU$ denote the normal subgroup of $G$ generated by $U$.
A smooth $R$-representation $\sigma$ of $L$ extends to a smooth representation $\e_G(\sigma)$ of $G$ trivial on $U$ if and only if it is trivial on $L \cap \GU$, in which case this extension is unique (\cite[§~II]{AHHV}).
We fix a parabolic subgroup $\Qb=\Mb\Nb$ of $\Gb$ such that $\Pb \subseteq \Qb$ and $\Lb \subseteq \Mb$, and we write $\bar \Qb=\Mb\bar \Nb$ for the opposite parabolic subgroup.
We let $\St_{\bar Q}^G(R)$ denote the generalised Steinberg representation of $G$ over $R$ relative to $\bar Q$ (\cite{GK} when $\Gb$ is split, \cite{LySt} in general).
We obtain an $R$-linear exact functor $\St_{\bar Q}^G : \Mod_{L/(L \cap \GU)}^\infty(R) \to \Mod_G^\infty(R)$ which commutes with small direct sums by setting $\St_{\bar Q}^G(\sigma) \coloneqq \e_G(\sigma) \otimes_R \St_{\bar Q}^G(R)$.

\subsection*{Results}

When $R$ is noetherian and $p$ is nilpotent in $R$, $\St_{\bar Q}^G$ respects admissibility (\cite[Theorem 4.21]{AHV}) and we prove that its restriction to admissible representations is fully faithful (Corollary \ref{coro:StFF}).
When $R$ is artinian and $p$ is nilpotent in $R$, we prove that $\St_{\bar Q}^G$ induces an isomorphism between the $R$-modules of Yoneda extensions (Proposition \ref{prop:Ext}).
We also extend the definition of $\St_{\bar Q}^G$ to $I$-adically continuous $R$-representations where $I$ is a finitely generated ideal of $R$.
When $R/I$ is noetherian and $p$ is nilpotent in $R/I$, $\St_{\bar Q}^G$ respects admissibility and we prove that its restriction to admissible representations is fully faithful (Proposition \ref{prop:StIFF}).
Finally, we prove that $\St_{\bar Q}^G$ induces bijections between certain sets of deformations (Propositions \ref{prop:StDef} and \ref{prop:StIDef}) and we deduce some results on universal deformation rings and universal deformations (Corollaries \ref{coro:pro-rep}, \ref{coro:Noe}, and \ref{coro:Lambda}).

\subsection*{Notations}

We keep the notations of the introduction.
We fix a minimal parabolic subgroup $\Bb \subseteq \Gb$ contained in $\Pb$ and a maximal split torus $\Sb \subseteq \Bb$ contained in $\Lb$.
We let $W$ (resp.\@ $W_\Mb$) be the Weyl group of $(\Gb,\Sb)$ (resp.\@ $(\Mb,\Sb)$) and $\Delta$ (resp.\@ $\Delta_\Mb$) be the set of simple roots of $(\Gb,\Bb,\Sb)$ (resp.\@ $(\Gb,\Mb' \cap \Bb,\Sb)$).
We let $\Zb_\Mb$ denote the centre of $\Mb$.
We write $\Mod_M^{\infty,Z_M-\lfin}(R)$ for the full subcategory of $\Mod_M^\infty(R)$ consisting of locally $Z_M$-finite representations (i.e.\@ those representations $\sigma$ such that for all $v \in \sigma$, $R[Z_M] \cdot v$ is contained in a finitely generated $R$-submodule).
Given an $R$-module $\sigma$, we set $\sigma_{p-\ord} \coloneqq \bigcap_{n \in \N} p^n\sigma$.

\section{\texorpdfstring{Extension from $L$ to $G$}{Extension from M to G}}

We do not make any assumption on $R$.
We recall the description of $\GU$ from \cite[§~II]{AHHV}.
Let $\iota : \Gb^\smc \twoheadrightarrow \Gb^\der \hookrightarrow \Gb$ be the simply connected cover of the derived subgroup of $\Gb$.
Recall that $\Gb^\smc$ is the direct product of its almost-simple components.
We let $\Bc$ be an indexing set for the isotropic almost-simple components of $\Gb^\smc$ and for $b \in \Bc$ we write $\Gbt_b$ for the corresponding component.
We let $\Bc(\Pb) \subseteq \Bc$ denote the subset consisting of those elements $b$ such that $\Gbt_b \not \subseteq \iota^{-1}(\Pb)$.
By \cite[II.5 Proposition]{AHHV}, we have $\GU = \iota(\prod_{b \in \Bc(\Pb)} \Gt_b)$.

\begin{lemm} \label{lemm:perf}
The group $\GU$ is perfect.
\end{lemm}

\begin{proof}
Since images and direct products of perfect groups are perfect, it is enough to prove that $\Gt_b$ is perfect for all $b \in \Bc(\Pb)$.
Let $b \in \Bc(\Pb)$.
We write $\Gt'_b$ for the (normal) subgroup of $\Gt_b$ generated by the $F$-points of the unipotent radicals of the parabolic subgroups of $\Gbt_b$.
Since $\Gbt_b$ is almost-simple and $\card(F) \geq 4$ (indeed $F$ is infinite), $\Gt'_b$ is perfect (see \cite[§~3.3]{Tits64}).
Since $\Gbt_b$ is almost-simple, simply connected, and isotropic, we have $\Gt_b=\Gt'_b$ (this is the Knesser--Tits conjecture, proved by Platonov over non-archimedean local fields, see \cite[Theorem 7.6]{PR}).
Thus $\Gt_b$ is perfect.
\end{proof}

Let $\sigma$ be an $R$-representation of $L$.
If $\sigma$ extends to an $R$-representation of $G$ trivial on $U$, then any extension has to be trivial on $\GU$.
In particular, $\sigma$ has to be trivial on $L \cap \GU$.
Since $L\GU=G$ (\cite[II.5 Corollary (iii)]{AHHV}), we deduce the converse: if $\sigma$ is trivial on $L \cap \GU$, then it extends uniquely to an $R$-representation $\e_G(\sigma)$ of $G$ trivial on $U$.
Moreover, $\e_G(\sigma)$ is smooth, admissible, or irreducible if and only if $\sigma$ is (see the proof of \cite[II.7 Proposition]{AHHV}).
Thus we obtain an $R$-linear fully faithful functor
\begin{equation*}
\e_G : \Mod_{L/(L \cap \GU)}^\infty(R) \to \Mod_G^\infty(R)
\end{equation*}
which commutes with all small limits and colimits (and in particular all finite ones, so that it is exact) and respects admissibility.

\section{Parabolic induction and ordinary parts}

We do not make any assumption on $R$.
Recall the smooth parabolic induction functor
\begin{equation*}
\Ind_{\bar Q}^G : \Mod_M^\infty(R) \to \Mod_G^\infty(R)
\end{equation*}
defined on any smooth $R$-representation $\sigma$ of $M$ as the $R$-module $\Ind_{\bar Q}^G(\sigma)$ of locally constant functions $f : G \to \sigma$ satisfying $f(m\bar ng) = m \cdot f(g)$ for all $m \in M$, $\bar n \in \bar N$, and $g \in G$, endowed with the smooth action of $G$ by right translation.
It is $R$-linear, exact, and commutes with small direct sums (\cite[Proposition 4.2]{VigAdj}).
In the other direction, there is the ordinary part functor (\cite{Em1,VigAdj})
\begin{equation*}
\Ord_Q : \Mod_G^\infty(R) \to \Mod_M^{\infty,Z_M-\lfin}(R)
\end{equation*}
which is right adjoint to the restriction of $\Ind_{\bar Q}^G$ to locally $Z_M$-finite representations (\cite[Corollary 7.3]{VigAdj}).
It is $R$-linear and left exact.
When $R$ is noetherian, $\Ord_Q$ also commutes with small inductive limits (\cite[Proposition 3.2.4]{Em1}) and both functors respect admissibility (\cite[Corollaries 4.7 and 8.3]{VigAdj}).

\medskip

Let $\sigma$ be a locally $Z_M$-finite smooth $R$-representation of $L$ trivial on $L \cap \GU$.
We assume $R$ noetherian and $\sigma_{p-\ord}=0$.
By \cite[Corollary 5.9]{AHV}, the unit of the adjunction between $\Ind_{\bar Q}^G$ and $\Ord_Q$ induces a natural isomorphism
\begin{equation} \label{unit}
\e_M(\sigma) \iso \Ord_Q(\Ind_{\bar Q}^G(\e_M(\sigma))).
\end{equation}
Given a smooth $R$-representation $\pi$ of $G$, the counit of the adjunction between $\Ind_{\bar Q}^G$ and $\Ord_Q$ induces a natural morphism
\begin{equation} \label{counit}
\Ind_{\bar Q}^G(\Ord_Q(\pi)) \to \pi.
\end{equation}
Applying $\Ind_{\bar Q}^G$ to \eqref{unit} and composing the result with \eqref{counit} with $\pi=\Ind_{\bar Q}^G(\e_M(\sigma))$ yields a natural composite
\begin{equation} \label{compInd}
\Ind_{\bar Q}^G(\e_M(\sigma)) \iso \Ind_{\bar Q}^G(\Ord_Q(\Ind_{\bar Q}^G(\e_M(\sigma)))) \to \Ind_{\bar Q}^G(\e_M(\sigma))
\end{equation}
which is the identity of $\Ind_{\bar Q}^G(\e_M(\sigma))$ by the unit-counit equations (\cite[(2.5)]{HSS}).

\section{Generalised Steinberg representations}

We do not make any assumption on $R$.
Given a smooth $R$-representation $\sigma$ of $L$ trivial on $L \cap \GU$, we define a smooth $R$-representation of $G$ by setting
\begin{equation*}
\St_{\bar Q}^G(\sigma) \coloneqq \frac{\Ind_{\bar Q}^G(\e_M(\sigma))}{\textstyle \sum_{\Qb' \supsetneq \Qb} \Ind_{\bar Q'}^G(\e_{M'}(\sigma))}
\end{equation*}
where $\Qb'=\Mb'\Nb'$ runs among the standard parabolic subgroups of $\Gb$ (i.e.\@ such that $\Bb \subseteq \Qb'$ and $\Sb \subseteq \Mb'$) strictly containing $\Qb$.
There are natural isomorphisms $\Ind_{\bar Q'}^G(\e_{M'}(\sigma)) \simeq \Ind_{\bar Q'}^G(R) \otimes_R \e_G(\sigma)$ for all standard parabolic subgroups $\Qb'=\Mb'\Nb'$ of $\Gb$ containing $\Qb$, which are compatible with the natural injections $\Ind_{\bar Q'}^G(\e_{M'}(\sigma)) \hookrightarrow \Ind_{\bar Q}^G(\e_M(\sigma))$, hence a natural isomorphism
\begin{equation} \label{St}
\St_{\bar Q}^G(\sigma) \simeq \St_{\bar Q}^G(R) \otimes_R \e_G(\sigma)
\end{equation}
and $\St_{\bar Q}^G(R)$ is free over $R$ (\cite[Corollaire 5.6]{LySt}).
Thus we obtain an $R$-linear exact functor
\begin{equation*}
\St_{\bar Q}^G : \Mod_{L/(L \cap \GU)}^\infty(R) \to \Mod_G^\infty(R)
\end{equation*}
which commutes with small direct sums.
When $R$ is noetherian and $p$ is nilpotent in $R$, $\St_{\bar Q}^G$ respects admissibility (\cite[Theorem 4.21]{AHV}).

\medskip

Let $\sigma$ be a locally $Z_M$-finite smooth $R$-representation of $L$ trivial on $L \cap \GU$.
We assume $R$ noetherian and $\sigma_{p-\ord}=0$.
By definition, there is a natural surjection
\begin{equation} \label{natsurj}
\Ind_{\bar Q}^G(\e_M(\sigma)) \twoheadrightarrow \St_{\bar Q}^G(\sigma).
\end{equation}
Applying $\Ord_Q$ to \eqref{natsurj} and precomposing the result with \eqref{unit} yields a natural composite
\begin{equation} \label{compOrd}
\e_M(\sigma) \iso \Ord_Q(\Ind_{\bar Q}^G(\e_M(\sigma))) \to \Ord_Q(\St_{\bar Q}^G(\sigma)).
\end{equation}

\begin{lemm} \label{lemm:compOrd}
The natural composite \eqref{compOrd} is an isomorphism.
\end{lemm}

\begin{proof}
We deduce from \cite[Lemma 7.10]{AHV} that there is a commutative diagram of smooth $R$-representations of $Q$
\begin{equation*}
\begin{tikzcd}
\cind_{\bar Q}^{\bar QQ}(\e_M(\sigma)) \dar[equal] \rar[hook] & \Ind_{\bar Q}^G(\e_M(\sigma)) \dar[two heads] \\
\cind_{\bar Q}^{\bar QQ}(\e_M(\sigma)) \rar[hook] & \St_{\bar Q}^G(\sigma)
\end{tikzcd}
\end{equation*}
where the right vertical arrow is \eqref{natsurj}.
We deduce from \cite[Corollary 7.14]{AHV} that applying $\Ord_Q$ to the horizontal arrows yields isomorphisms.
Thus applying $\Ord_Q$ to \eqref{natsurj} also yields an isomorphism.
\end{proof}

Applying $\Ind_{\bar Q}^G$ to \eqref{compOrd} and composing the result with \eqref{counit} with $\pi=\St_{\bar Q}^G(\sigma)$ yields a natural composite
\begin{equation} \label{compSt}
\Ind_{\bar Q}^G(\e_M(\sigma)) \iso \Ind_{\bar Q}^G(\Ord_Q(\St_{\bar Q}^G(\sigma))) \to \St_{\bar Q}^G(\sigma).
\end{equation}

\begin{lemm} \label{lemm:compSt}
The natural composite \eqref{compSt} is the natural surjection \eqref{natsurj}.
\end{lemm}

\begin{proof}
There is a commutative diagram of smooth $R$-representations of $G$
\begin{equation*}
\begin{tikzcd}
\Ind_{\bar Q}^G(\e_M(\sigma)) \dar[equal] \rar{\sim} & \Ind_{\bar Q}^G(\Ord_Q(\Ind_{\bar Q}^G(\e_M(\sigma)))) \dar{\vertiso} \rar{\sim} & \Ind_{\bar Q}^G(\e_M(\sigma)) \dar[two heads] \\
\Ind_{\bar Q}^G(\e_M(\sigma)) \rar{\sim} & \Ind_{\bar Q}^G(\Ord_Q(\St_{\bar Q}^G(\sigma))) \rar[two heads] & \St_{\bar Q}^G(\sigma)
\end{tikzcd}
\end{equation*}
where the upper and lower arrows are \eqref{compInd} and \eqref{compSt} respectively, the middle vertical arrow is obtained by applying $\Ind_{\bar Q}^G \Ord_Q$ to \eqref{natsurj}, and the right vertical arrow is \eqref{natsurj}.
The left square is commutative by definition and the right square is commutative by naturality of \eqref{counit}.
Since the upper horizontal composite is the identity of $\Ind_{\bar Q}^G(\e_M(\sigma))$, the lower horizontal composite is \eqref{natsurj}.
\end{proof}

\begin{prop} \label{prop:StFF}
Assume that $R$ is noetherian and $p$ is nilpotent in $R$.
The functor $\St_{\bar Q}^G : \Mod_{L/(L \cap \GU)}^{\infty,Z_M-\lfin}(R) \to \Mod_G^\infty(R)$ is fully faithful.
\end{prop}

\begin{proof}
We do not assume $p$ nilpotent in $R$.
Let $\sigma$ and $\sigma'$ be two locally $Z_M$-finite $R$-representations of $L$ trivial on $L \cap \GU$.
Let $\phi : \sigma \to \sigma'$ be a morphism such that $\St_{\bar Q}^G(\phi)=0$.
Using \eqref{St}, we deduce that $\e_G(\phi)=0$ because $\St_{\bar Q}^G(R)$ is free over $R$, hence $\phi=0$ because $\e_G$ is faithful.
Thus $\St_{\bar Q}^G$ is faithful.
Assume $\sigma_{p-\ord}=\sigma'_{p-\ord}=0$ (this is automatically true if $p$ is nilpotent in $R$).
Let $\psi : \St_{\bar Q}^G(\sigma) \to \St_{\bar Q}^G(\sigma')$ be a morphism.
There is a commutative diagram of smooth $R$-representations of $M$
\begin{equation} \label{phi}
\begin{tikzcd}[column sep=5em]
\e_M(\sigma) \dar{\vertiso} \rar{\e_M(\phi)} & \e_M(\sigma') \dar{\vertiso} \\
\Ord_Q(\St_{\bar Q}^G(\sigma)) \rar{\Ord_Q(\psi)} & \Ord_Q(\St_{\bar Q}^G(\sigma'))
\end{tikzcd}
\end{equation}
where the vertical arrows are \eqref{compOrd} and its analogue for $\sigma'$, which are isomorphisms by Lemma \ref{lemm:compOrd}, and $\phi : \sigma \to \sigma'$ is the unique morphism making the diagram commutative.
There is a commutative diagram of smooth $R$-representations of $G$
\begin{equation*}
\begin{tikzcd}[column sep=7em]
\Ind_{\bar Q}^G(\e_M(\sigma)) \dar{\vertiso} \rar{\Ind_{\bar Q}^G(\phi)} & \Ind_{\bar Q}^G(\e_M(\sigma')) \dar{\vertiso} \\
\Ind_{\bar Q}^G(\Ord_Q(\St_{\bar Q}^G(\sigma))) \dar[two heads] \rar{\Ind_{\bar Q}^G(\Ord_Q(\psi))} & \Ind_{\bar Q}^G(\Ord_Q(\St_{\bar Q}^G(\sigma'))) \dar[two heads] \\
\St_{\bar Q}^G(\sigma) \rar{\psi} & \St_{\bar Q}^G(\sigma')
\end{tikzcd}
\end{equation*}
where the upper square is obtained by applying $\Ind_{\bar Q}^G$ to \eqref{phi} and the lower vertical arrows are \eqref{counit} with $\pi=\St_{\bar Q}^G(\sigma)$ and its analogue for $\sigma'$.
The lower square is commutative by naturality of \eqref{counit}.
By Lemma \ref{lemm:compSt}, the vertical composites are \eqref{natsurj} and its analogue for $\sigma'$.
Thus $\psi=\St_{\bar Q}^G(\phi)$.
\end{proof}

\begin{coro} \label{coro:StFF}
Assume that $R$ is noetherian and $p$ is nilpotent in $R$.
The functor $\St_{\bar Q}^G : \Mod_{L/(L \cap \GU)}^\adm(R) \to \Mod_G^\adm(R)$ is fully faithful.
\end{coro}

\begin{proof}
There is a commutative diagram of $R$-linear categories
\begin{equation*}
\begin{tikzcd}
\Mod_{L/(L \cap \GU)}^{\infty,Z_M-\lfin}(R) \rar{\St_{\bar Q}^G} & \Mod_G^\infty(R) \\
\Mod_{L/(L \cap \GU)}^\adm(R) \uar[hook] \rar{\St_{\bar Q}^G} & \Mod_G^\adm(R) \uar[hook]
\end{tikzcd}
\end{equation*}
where the vertical arrows are the full inclusions (see \cite[Lemma 3.6]{VigAdj} for the left one).
Since the upper horizontal arrow is fully faithful by Proposition \ref{prop:StFF}, so is the lower horizontal arrow.
\end{proof}

\section{Bruhat filtration}

We do not make any assumption on $R$.
Let $\sigma$ be a smooth $R$-representation of $L$ trivial on $L \cap \GU$.
We write $\prescript{\Qb}{}{W}$ for the set of representatives of minimal length of the right cosets $W_\Mb \backslash W$ and $w_{\Mb,0}$ for the longest element in $W_\Mb$.
We set
\begin{equation*}
\prescript{\Qb}{0}{W} \coloneqq \left\{ \prescript{\Qb}{}{w} \in \prescript{\Qb}{}{W} \middlevert w_{\Mb,0} \prescript{\Qb}{}{w} \not \in w_{\Mb',0} \prescript{\Qb'}{}{W} \text{ for all } \Qb' \supsetneq \Qb \right\}
\end{equation*}
where $\Qb'=\Mb'\Nb'$ runs among the standard parabolic subgroups of $\Gb$ strictly containing $\Qb$.
In \cite[§~2.2]{JHD} is constructed a natural filtration $\Fil_B^\bullet(\Ind_{\bar Q}^G(\e_M(\sigma)))$ of $\Ind_{\bar Q}^G(\e_M(\sigma))_{|B}$ by $R$-subrepresentations indexed by $\prescript{\Qb}{}{W}$ such that for all $\prescript{\Qb}{}{w} \in \prescript{\Qb}{}{W}$, there is a natural isomorphism
\begin{equation*}
\Gr_B^{\prescript{\Qb}{}{w}}(\Ind_{\bar Q}^G(\e_M(\sigma))) \simeq \cind_{\bar Q}^{\bar Q \prescript{\Qb}{}{w} B}(\e_M(\sigma)).
\end{equation*}
By taking its image by \eqref{natsurj}, we obtain a natural filtration $\Fil_B^\bullet(\St_{\bar Q}^G(\sigma))$ of $\St_{\bar Q}^G(\sigma)_{|B}$ by $R$-subrepresentations indexed by $\prescript{\Qb}{}{W}$.
We deduce from \cite[Lemma 7.10]{AHV} that for all $\prescript{\Qb}{}{w} \in \prescript{\Qb}{}{W}$, there is a natural isomorphism
\begin{equation} \label{Gr}
\Gr_B^{\prescript{\Qb}{}{w}}(\St_{\bar Q}^G(\sigma)) \simeq
\begin{cases}
\cind_{\bar Q}^{\bar Q \prescript{\Qb}{}{w} B}(\e_M(\sigma)) &\text{if $\prescript{\Qb}{}{w} \in \prescript{\Qb}{0}{W}$,} \\
0 &\text{otherwise.}
\end{cases}
\end{equation}

\section{Higher ordinary parts}

We assume $R$ artinian, $p$ nilpotent in $R$, and $\car(F)=0$.
The results of \cite{Em2} remain valid over $R$ instead of $A$ (in loc.\@ cit.\@ $A$ is an artinian local $\Z_p$-algebra with finite residue field; here $R$ is naturally an artinian $\Z_p$-algebra but need not be local with finite residue field).
In particular, there are $R$-linear functors
\begin{equation*}
\HOrd[n]_Q : \Mod_G^\infty(R) \to \Mod_M^{\infty,Z_M-\lfin}(R)
\end{equation*}
for all $n \in \N$ with $\HOrd[0]_Q=\Ord_Q$.
They commute with small inductive limits, respect admissibility, and form a cohomological $\delta$-functor
\begin{equation*}
\HOrd_Q : \Mod_G^\adm(R) \to \Mod_M^\adm(R).
\end{equation*}
Likewise, the results of \cite{JHD} remain valid over $R$ instead of $A$.

\begin{prop} \label{prop:H1Ord}
Let $\sigma$ be an admissible $R$-representation of $L$ trivial on $L \cap \GU$ and $\Qb'=\Mb'\Nb'$ be a standard parabolic subgroup of $\Gb$ containing $\Qb$.
We have $\HOrd[1]_{Q'}(\St_{\bar Q}^G(\sigma))=0$.
\end{prop}

\begin{proof}
Note that $\HOrd_Q$ restricts to $\Mod_B^\infty(R)$ (see \cite[Definition 3.1.4]{JHD}).
Using \eqref{Gr}, it is enough to prove that for all $\prescript{\Qb}{}{w} \in \prescript{\Qb}{0}{W}$,
\begin{equation*}
\HOrd[1]_{Q'}(\cind_{\bar Q}^{\bar Q \prescript{\Qb}{}{w} B}(\e_M(\sigma)))=0.
\end{equation*}
Let $\prescript{\Qb}{}{w} \in \prescript{\Qb}{0}{W}$ and write $\prescript{\Qb}{}{w} = \prescript{\Qb}{}{w}^{\Qb'} w_{\Mb'}$ with $\prescript{\Qb}{}{w}^{\Qb'} \in \prescript{\Qb}{}{W}^{\Qb'} \coloneqq \prescript{\Qb}{}{W} \cap \prescript{\Qb'}{}{W}^{-1}$ and $w_{\Mb'} \in W_{\Mb'}$ (see \cite[§~2.1]{JHD}).
Let $\Qb''=\Mb''\Nb''$ be the standard parabolic subgroup of $\Gb$ such that $\Delta_{\Mb''} = \Delta_\Mb \cap \prescript{\Qb}{}{w}^{\Qb'}(\Delta_{\Mb'})$.
Using \cite[Theorem 3.3.3]{JHD}, it is enough to prove that
\begin{equation*}
\HOrd[1-{[F:\Q_p]}d_{\prescript{\Qb}{}{w}^{\Qb'}}]_{M \cap Q''}(\e_M(\sigma))=0
\end{equation*}
(see \cite[Notation 2.3.3]{JHD} for the definition of $d_{\prescript{\Qb}{}{w}^{\Qb'}} \in \N$).
If $d_{\prescript{\Qb}{}{w}^{\Qb'}}=0$, i.e.\@ $\prescript{\Qb}{}{w}^{\Qb'}=1$, then $\HOrd[1]_{M \cap Q''}=\HOrd[1]_M=0$ by \cite[Lemma 3.6.1]{Em2}.
If either $F=\Q_p$ and $d_{\prescript{\Qb}{}{w}^{\Qb'}}>1$, or $F \neq \Q_p$ and $d_{\prescript{\Qb}{}{w}^{\Qb'}}>0$, then $1-[F:\Q_p]d_{\prescript{\Qb}{}{w}^{\Qb'}}<0$ so that $\HOrd[1-{[F:\Q_p]}d_{\prescript{\Qb}{}{w}^{\Qb'}}]_{M \cap Q''}=0$.
Now assume $F=\Q_p$ and $d_{\prescript{\Qb}{}{w}^{\Qb'}}=1$.
Thus $\prescript{\Qb}{}{w}^{\Qb'}=s_\alpha$ for some $\alpha \in \Delta \backslash \Delta_{\Mb'}$ and $\Delta_{\Mb''}=\Delta_\Mb \cap \{\alpha\}^\perp$.
If $\alpha \perp \Delta_\Mb$, then $w_{\Mb,0}s_\alpha=w_{\Mb_\alpha,0}$, where $\Qb_\alpha=\Mb_\alpha\Nb_\alpha$ is the standard parabolic subgroup of $\Gb$ such that $\Delta_{\Mb_\alpha}=\Delta_\Mb \sqcup \{\alpha\}$, and $w_{\Mb'} \in \prescript{\Qb_\alpha}{}{W}$ (because $w_{\Mb'} \in \prescript{\Mb' \cap \Qb}{}{W}_{\Mb'}$ and $\alpha \not \in \Delta_{\Mb'}$) so that $w_{\Mb,0}\prescript{\Qb}{}{w} \in w_{\Mb_\alpha,0} \prescript{\Qb_\alpha}{}{W}$, which contradicts the fact that $\prescript{\Qb}{}{w} \in \prescript{\Qb}{0}{W}$.
Thus $\alpha \not \perp \Delta_\Mb$ so that $\Mb \cap \Qb'' \subsetneq \Mb$.
Since $\e_M(\sigma)$ is trivial on $M \cap N''$ (see the proof of \cite[II.7 Corollary 2]{AHHV}), we deduce that $\Ord_{M \cap Q''}(\e_M(\sigma))=0$.
\end{proof}

\section{Extensions}

We do not make any assumption on $R$.
Given two smooth $R$-representations $\pi,\pi'$ of $G$, we write $\Ext_G^1(\pi',\pi)$ for the $R$-module of extensions of $\pi'$ by $\pi$ in $\Mod_G^\infty(R)$.

\begin{lemm} \label{lemm:Ext}
Let $\sigma,\sigma'$ be smooth $R$-representations of $L$ trivial on $L \cap \GU$.
The functor $\e_G$ induces an $R$-linear isomorphism
\begin{equation*}
\Ext_{L/(L \cap \GU)}^1(\sigma',\sigma) \iso \Ext_G^1(\e_G(\sigma'),\e_G(\sigma)).
\end{equation*}
\end{lemm}

\begin{proof}
The morphism is well defined since $\e_G$ is exact.
It is $R$-linear since $\e_G$ is, and it is injective since $\e_G$ is fully faithful.
We prove that it is surjective.
Let
\begin{equation*}
0 \to \e_G(\sigma) \to \pi \to \e_G(\sigma') \to 0
\end{equation*}
be a short exact sequence of smooth $R$-representations of $G$.
We prove that $\pi$ is trivial on $\GU$ so that $\pi=\e_G(\pi_{|L})$.
Taking the $\GU$-invariants yields an exact sequence of $R$-modules
\begin{equation} \label{GUinv}
0 \to \e_G(\sigma) \to \pi^{\GU} \to \e_G(\sigma') \to \Hr^1(\GU,\e_G(\sigma))
\end{equation}
where the rightmost term is the $R$-module of $\GU$-cohomology computed using locally constant cochains (see \cite[§~2.2]{Em2}).
Since $\GU$ acts trivially on $\e_G(\sigma)$, there is a natural isomorphism
\begin{equation*}
\Hr^1(\GU,\e_G(\sigma)) \simeq \Hom_\Grp^\cont(\GU,\e_G(\sigma))
\end{equation*}
where the right-hand side is the $R$-module of continuous group homomorphisms, which is trivial by Lemma \ref{lemm:perf} (because $\e_G(\sigma)$ is commutative).
Using \eqref{GUinv}, we deduce that $\pi^{\GU}=\pi$.
\end{proof}

\begin{prop} \label{prop:Ext}
Assume that $R$ is artinian and $p$ is nilpotent in $R$.
Let $\sigma,\sigma'$ be admissible $R$-representations of $L$ trivial on $L \cap \GU$.
The functor $\St_{\bar Q}^G$ induces an $R$-linear isomorphism
\begin{equation*}
\Ext_{L/(L \cap \GU)}^1(\sigma',\sigma) \iso \Ext_G^1(\St_{\bar Q}^G(\sigma'),\St_{\bar Q}^G(\sigma)).
\end{equation*}
\end{prop}

\begin{proof}
The morphism is well defined since $\St_{\bar Q}^G$ is exact.
It is $R$-linear since $\St_{\bar Q}^G$ is, and it is injective since $\St_{\bar Q}^G$ is fully faithful by Corollary \ref{coro:StFF}.
We prove that it is surjective by induction on $|\Delta \backslash \Delta_\Mb|$.
If $\Delta_\Mb=\Delta$ (i.e.\@ $\Qb=\Mb=\Gb$), then the result is Lemma \ref{lemm:Ext}.
We assume $\Delta_\Mb \neq \Delta$ and that we know the result for relatively bigger parabolic subgroups.
Let
\begin{equation} \label{extpi}
0 \to \St_{\bar Q}^G(\sigma) \to \pi \to \St_{\bar Q}^G(\sigma') \to 0
\end{equation}
be a short exact sequence of admissible $R$-representations of $G$.
Pick $\alpha \in \Delta \backslash \Delta_\Mb$ and let $\Qb^\alpha=\Mb^\alpha\Nb^\alpha$ and $\Qb_\alpha=\Mb_\alpha\Nb_\alpha$ be the standard parabolic subgroups of $\Gb$ defined by $\Delta_{\Mb^\alpha} = \Delta \backslash \{\alpha\}$ and $\Delta_{\Mb_\alpha} = \Delta_\Mb \sqcup \{\alpha\}$ respectively.
We have $|\Delta_{\Mb^\alpha} \backslash \Delta_\Mb| = |\Delta \backslash \Delta_{\Mb_\alpha}| = |\Delta \backslash \Delta_\Mb|-1$.

Applying $\Ord_{Q^\alpha}$ to \eqref{extpi} and using \cite[Theorem 6.1 (ii)]{AHV} and Proposition \ref{prop:H1Ord} (if $\car(F)=0$) or \cite[Theorem 1]{JHE} (if $\car(F)=p$) yields a short exact sequence of admissible $R$-representations of $M^\alpha$
\begin{equation} \label{extOrda}
0 \to \St_{M^\alpha \cap \bar Q}^{M^\alpha}(\sigma) \to \Ord_{Q^\alpha}(\pi) \to \St_{M^\alpha \cap \bar Q}^{M^\alpha}(\sigma') \to 0.
\end{equation}
By the induction hypothesis, there exists a short exact sequence of admissible $R$-representations of $L/(L \cap \langle\prescript{M^\alpha}{}{(M^\alpha \cap U)}\rangle)$
\begin{equation} \label{exteta}
0 \to \sigma \to \eta \to \sigma' \to 0
\end{equation}
and a commutative diagram of admissible $R$-representations of $M^\alpha$
\begin{equation} \label{extOrdaiso}
\begin{tikzcd}
0 \rar & \St_{M^\alpha \cap \bar Q}^{M^\alpha}(\sigma) \dar[equal] \rar & \St_{M^\alpha \cap \bar Q}^{M^\alpha}(\eta) \dar{\vertiso} \rar & \St_{M^\alpha \cap \bar Q}^{M^\alpha}(\sigma') \dar[equal] \rar & 0 \\
0 \rar & \St_{M^\alpha \cap \bar Q}^{M^\alpha}(\sigma) \rar & \Ord_{Q^\alpha}(\pi) \rar & \St_{M^\alpha \cap \bar Q}^{M^\alpha}(\sigma') \rar & 0
\end{tikzcd}
\end{equation}
where the upper row is obtained by applying $\St_{M^\alpha \cap \bar Q}^{M^\alpha}$ to \eqref{exteta} and the lower row is \eqref{extOrda}.

Now, we have a commutative diagram of admissible $R$-representations of $G$
\begin{equation} \label{diag}
\begin{tikzcd}[column sep=tiny]
& 0 \dar & 0 \dar & 0 \dar & \\
0 \rar & \St_{\bar Q_\alpha}^G(\sigma) \dar \rar & \tau \dar \rar & \St_{\bar Q_\alpha}^G(\sigma') \dar \rar & 0 \\
0 \rar & \Ind_{\bar Q^\alpha}^G(\St_{M^\alpha \cap \bar Q}^{M^\alpha}(\sigma)) \dar \rar & \Ind_{\bar Q^\alpha}^G(\St_{M^\alpha \cap \bar Q}^{M^\alpha}(\eta)) \dar \rar & \Ind_{\bar Q^\alpha}^G(\St_{M^\alpha \cap \bar Q}^{M^\alpha}(\sigma')) \dar \rar & 0 \\
0 \rar & \St_{\bar Q}^G(\sigma) \rar \dar & \pi \rar \dar & \St_{\bar Q}^G(\sigma') \rar \dar & 0 \\
& 0 & 0 & 0 &
\end{tikzcd}
\end{equation}
whose rows and columns are exact.
The lower row is \eqref{extpi} and the middle row is obtained by applying $\Ind_{\bar Q^\alpha}^G$ to \eqref{extOrda} and using the middle vertical isomorphism of \eqref{extOrdaiso}.
The lower vertical arrows are induced by the counit of the adjunction between $\Ind_{\bar Q^\alpha}^G$ and $\Ord_{Q^\alpha}$ (the left and right ones are the natural surjections by an analogue of Lemma \ref{lemm:compSt} and the middle one is surjective by the five lemma).
The upper vertical arrows are the kernels of the lower vertical arrows (the left and right ones are the natural injections).

By the induction hypothesis, there exists a short exact sequence of admissible $R$-representations of $L/(L \cap \GU)$
\begin{equation} \label{extkappa}
0 \to \sigma \to \kappa \to \sigma' \to 0
\end{equation}
and a commutative diagram of admissible $R$-representations of $G$
\begin{equation} \label{exttauiso}
\begin{tikzcd}
0 \rar & \St_{\bar Q}^G(\sigma) \dar[equal] \rar & \St_{\bar Q}^G(\kappa) \dar{\vertiso} \rar & \St_{\bar Q}^G(\sigma') \dar[equal] \rar & 0 \\
0 \rar & \St_{\bar Q}^G(\sigma) \rar & \tau \rar & \St_{\bar Q}^G(\sigma') \rar & 0
\end{tikzcd}
\end{equation}
where the upper row is obtained by applying $\St_{M^\alpha \cap \bar Q}^{M^\alpha}$ to \eqref{extkappa} and the lower row is the upper row of \eqref{diag}.
Combining \eqref{exttauiso} with the upper and middle rows of \eqref{diag} yields a commutative diagram of admissible $R$-representations of $G$
\begin{equation*}
\begin{tikzcd}[column sep=small]
0 \rar & \St_{\bar Q}^G(\sigma) \dar[hook] \rar & \St_{\bar Q}^G(\kappa) \dar[hook] \rar & \St_{\bar Q}^G(\sigma') \dar[hook] \rar & 0 \\
0 \rar & \Ind_{\bar Q^\alpha}^G(\St_{M^\alpha \cap \bar Q}^{M^\alpha}(\sigma)) \rar & \Ind_{\bar Q^\alpha}^G(\St_{M^\alpha \cap \bar Q}^{M^\alpha}(\eta)) \rar & \Ind_{\bar Q^\alpha}^G(\St_{M^\alpha \cap \bar Q}^{M^\alpha}(\sigma')) \rar & 0
\end{tikzcd}
\end{equation*}
and passing to the $\bar N^\alpha$-coinvariants yields a commutative diagram of admissible $R$-representations of $M^\alpha$
\begin{equation*}
\begin{tikzcd}
0 \rar & \St_{M^\alpha \cap \bar Q}^{M^\alpha}(\sigma) \dar[equal] \rar & \St_{M^\alpha \cap \bar Q}^{M^\alpha}(\kappa) \dar{\vertiso} \rar & \St_{M^\alpha \cap \bar Q}^{M^\alpha}(\sigma') \dar[equal] \rar & 0 \\
0 \rar & \St_{M^\alpha \cap \bar Q}^{M^\alpha}(\sigma) \rar & \St_{M^\alpha \cap \bar Q}^{M^\alpha}(\eta) \rar & \St_{M^\alpha \cap \bar Q}^{M^\alpha}(\sigma') \rar & 0
\end{tikzcd}
\end{equation*}
whose rows are exact by \cite[Corollary 6.2 (i)]{AHV} (the left and right vertical morphisms are the identities by taking a closer look at \cite[§~6.2]{AHV} and the middle one is an isomorphism by the five lemma).
By the injectivity part of the statement, it is induced by applying $\St_{M^\alpha \cap \bar Q}^{M^\alpha}$ to a commutative diagram of admissible $R$-representations of $L/(L \cap \GU)$
\begin{equation*}
\begin{tikzcd}
0 \rar & \sigma \dar[equal] \rar & \kappa \dar{\vertiso} \rar & \sigma' \dar[equal] \rar & 0 \\
0 \rar & \sigma \rar & \eta \rar & \sigma' \rar & 0
\end{tikzcd}
\end{equation*}
where the upper row is \eqref{extkappa} and the lower row is \eqref{exteta}.
Applying $\St_{\bar Q}^G$ to the middle vertical isomorphism and precomposing the result with the inverse of the middle vertical isomorphism of \eqref{exttauiso} yields an isomorphism $\tau \simeq \St_{\bar Q}^G(\eta)$ via which the upper middle vertical arrow of \eqref{diag} is the natural injection.
Indeed, there are natural isomorphisms
\begin{align*}
\Hom_G(\St_{\bar Q}^G(\kappa),\Ind_{\bar Q^\alpha}(\St_{M^\alpha \cap \bar Q}^{M^\alpha}(\eta))) &\simeq \Hom_{M^\alpha}(\St_{M^\alpha \cap \bar Q}^{M^\alpha}(\kappa),\St_{M^\alpha \cap \bar Q}^{M^\alpha}(\eta)) \\
&\simeq \Hom_L(\kappa,\eta)
\end{align*}
(the first one is induced by the adjunction between $(-)_{\bar N^\alpha}$ and $\Ind_{\bar Q^\alpha}^G$ together with \cite[Theorem 6.1 (i)]{AHV} and the second one is follows from Corollary \ref{coro:StFF}).

We conclude that the middle and lower rows of \eqref{diag} yield a commutative diagram of admissible $R$-representations of $G$
\begin{equation*}
\begin{tikzcd}
0 \rar & \St_{\bar Q}^G(\sigma) \dar[equal] \rar & \St_{\bar Q}^G(\eta) \dar{\vertiso} \rar & \St_{\bar Q}^G(\sigma') \dar[equal] \rar & 0 \\
0 \rar & \St_{\bar Q}^G(\sigma) \rar & \pi \rar & \St_{\bar Q}^G(\sigma') \rar & 0
\end{tikzcd}
\end{equation*}
where the upper row is obtained by applying $\St_{\bar Q}^G$ to \eqref{exteta} and the lower row is \eqref{extpi}.
\end{proof}

\section{\texorpdfstring{$I$-adically continuous representations}{I-adically continuous representations}}

Let $I$ be an ideal of $R$.
We write $\Mod_G^{I-\cont}(R)$ for the category of $I$-adically continuous $R$-representations of $G$ (i.e.\@ $I$-adically complete and separated $R[G]$-modules $\pi$ such that the map $G \times \pi \to \pi$ is jointly continuous when $\pi$ is given its $I$-adic topology, or equivalently the $R/I^n$-representations $\pi/I^n\pi$ of $G$ are smooth for all $n \geq 1$) and $R[G]$-linear maps.
We write $\Mod_G^{I-\adm}(R)$ for the full subcategory of $\Mod_G^{I-\cont}(R)$ consisting of admissible representations (i.e.\@ those representations $\pi$ such that the smooth $R/I^n$-representations $\pi/I^n\pi$ of $G$ are admissible for all $n \geq 1$).
If $I$ is nilpotent, then $\Mod_G^{I-\cont}(R)=\Mod_G^\infty(R)$ and $\Mod_G^{I-\adm}(R)=\Mod_G^\adm(R)$.

We assume $I$ finitely generated.
Given an $I$-adically continuous $R$-representation $\sigma$ of $L$ trivial on $L \cap \GU$, we set
\begin{equation*}
\St_{\bar Q}^G(\sigma) \coloneqq \varprojlim_{n \geq 1} \St_{\bar Q}^G(\sigma/I^n\sigma).
\end{equation*}
Using \eqref{St} with $R/I^n$ and $\sigma/I^n\sigma$ instead of $R$ and $\sigma$ respectively for all $n \geq 1$, we see that $(\St_{\bar Q}^G(\sigma/I^n\sigma))_{n \geq 1}$ is an $I$-adic system of $R$-modules (\cite[Definition 2.1]{Yek2}).
By \cite[Theorem 2.8]{Yek2}, $\St_{\bar Q}^G(\sigma)$ is an $I$-adically continuous $R$-representation of $G$ and there are isomorphisms
\begin{equation} \label{StI}
\St_{\bar Q}^G(\sigma) / I^n \St_{\bar Q}^G(\sigma) \simeq \St_{\bar Q}^G(\sigma / I^n \sigma)
\end{equation}
for all $n \geq 1$.
Thus we obtain an $R$-linear functor
\begin{equation*}
\St_{\bar Q}^G : \Mod_{L/(L \cap \GU)}^{I-\cont}(R) \to \Mod_G^{I-\cont}(R)
\end{equation*}
which respects admissibility when $R/I$ is noetherian and $p$ is nilpotent in $R/I$.

\begin{prop} \label{prop:StIFF}
Assume that $I$ is finitely generated, $R/I$ is noetherian, and $p$ is nilpotent in $R/I$.
The functor $\St_{\bar Q}^G : \Mod_{L/(L \cap \GU)}^{I-\adm}(R) \to \Mod_G^{I-\adm}(R)$ is fully faithful.
\end{prop}

\begin{proof}
By induction, $R/I^n$ is noetherian and $p$ is nilpotent in $R/I^n$ for all $n \geq 1$.
There are natural $R$-linear isomorphisms
\begin{align*}
\Hom_G(\St_{\bar Q}^G(\sigma'),\St_{\bar Q}^G(\sigma)) &\simeq \varprojlim_{n \geq 1} \Hom_G(\St_{\bar Q}^G(\sigma')/I^n\St_{\bar Q}^G(\sigma'),\St_{\bar Q}^G(\sigma)/I^n\St_{\bar Q}^G(\sigma)) \\
&\simeq \varprojlim_{n \geq 1} \Hom_G(\St_{\bar Q}^G(\sigma'/I^n\sigma'),\St_{\bar Q}^G(\sigma/I^n\sigma)) \\
&\simeq \varprojlim_{n \geq 1} \Hom_L(\sigma'/I^n\sigma',\sigma/I^n\sigma) \\
&\simeq \Hom_L(\sigma',\sigma)
\end{align*}
(the first and last ones follow from the fact that $\St_{\bar Q}^G(\sigma)$ and $\sigma$ are $I$-adically complete and separated, the second one follows from \eqref{StI} and its analogue for $\sigma'$, and the third one follows from Corollary \ref{coro:StFF} with $R/I^n$ instead of $R$ for all $n \geq 1$).
\end{proof}

\section{Artinian deformations}

Let $E$ be a finite extension of $\Q_p$.
We let $\Oc$ denote its ring of integers and $k$ denote its residue field.
Given a local $\Oc$-algebra $A$, we let $\mf_A$ denote its maximal ideal.
We write $\Art(\Oc)$ for the category of artinian local $\Oc$-algebras $A$ such that the structural morphism $\Oc \to A$ induces an isomorphism $k \iso A/\mf_A$.
The morphisms $A \to A'$ are the (local) $\Oc$-algebra homomorphisms.

A lift of a smooth $k$-representation $\bar\pi$ of $G$ over $A \in \Art(\Oc)$ is a pair $(\pi,\phi)$ where $\pi$ is a smooth $A$-representation of $G$ free over $A$ and $\phi : \pi \twoheadrightarrow \bar\pi$ is a surjection with kernel $\mf_A\pi$.
A morphism $\iota : (\pi,\phi) \to (\pi',\phi')$ is a morphism $\iota : \pi \to \pi'$ such that $\phi = \phi' \circ \iota$.
For $A \in \Art(\Oc)$, we let $\Def_{\bar\pi}(A)$ denote the set of isomorphism classes of lifts of $\bar\pi$ over $A$.
If $(\pi,\phi)$ is a lift of $\bar\pi$ over $A \in \Art(\Oc)$ and $A \to A'$ is a morphism in $\Art(\Oc)$, then $(\pi \otimes_A A',\phi')$ is naturally a lift of $\bar\pi$ over $A'$, where $\phi' : \pi \otimes_A A' \twoheadrightarrow \bar\pi$ is the morphism induced by $\phi$.
Thus we obtain a functor $\Def_{\bar\pi} : \Art(\Oc) \to \Set$.

Let $\bar\sigma$ be a smooth $k$-representation of $L$ trivial on $L \cap \GU$ and set $\bar\pi \coloneqq \St_{\bar Q}^G(\bar\sigma)$.
For $A \in \Art(\Oc)$, we consider the set of deformations $\Def_{\bar\sigma}(A)$ of $\bar\sigma$ over $A$ as a smooth $k$-representation of $L/(L \cap \GU)$.
We obtain a functor $\Def_{\bar\sigma} : \Art(\Oc) \to \Set$.
Using \eqref{St} with $R=A$ and the fact that $\St_{\bar Q}^G(A)$ is free over $A$, we see that:
\begin{itemize}
\item $\St_{\bar Q}^G(\sigma)$ is free over $A$ if and only if $\sigma$ is,
\item $\St_{\bar Q}^G$ is compatible with base change: for any morphism $A \to A'$ in $\Art(\Oc)$, there is a natural isomorphism $\St_{\bar Q}^G(\sigma) \otimes_A A' \simeq \St_{\bar Q}^G(\sigma \otimes_A A')$.
\end{itemize}
Thus the functor $\St_{\bar Q}^G$ induces a natural transformation
\begin{equation} \label{StDef}
\St_{\bar Q}^G : \Def_{\bar\sigma} \to \Def_{\bar\pi}.
\end{equation}

\begin{prop} \label{prop:StDef}
If $\bar\sigma$ is admissible, then \eqref{StDef} is an isomorphism of functors $\Art(\Oc) \to \Set$.
\end{prop}

\begin{proof}
Let $A \in \Art(\Oc)$.
We prove that \eqref{StDef} induces a bijection between the $A$-points.

We start with injectivity.
Assume that $\bar\sigma$ is locally $Z_M$-finite (this is automatically true if $\bar\sigma$ is admissible, see \cite[Lemma 3.6]{VigAdj}).
Let $(\sigma,\psi)$ and $(\sigma',\psi')$ be two lifts of $\bar\sigma$ over $A$ such that there exists an isomorphism $\imath : \St_{\bar Q}^G(\sigma) \iso \St_{\bar Q}^G(\sigma')$ such that $\St_{\bar Q}^G(\psi) = \St_{\bar Q}^G(\psi') \circ \imath$.
By \cite[Lemma 2.12 (1)]{HSS}, $\sigma$ and $\sigma'$ are also locally $Z_M$-finite.
Thus applying $\Ord_Q$ and using \eqref{compOrd} and its analogue for $\sigma'$ yields an isomorphism $\jmath : \sigma \iso \sigma'$ such that $\psi = \psi' \circ \jmath$.

We now turn to surjectivity.
Assume that $\bar\sigma$ is admissible.
We proceed by induction on the length of $A$.
The base case $A=k$ is trivial.
Assume $A \neq k$ and that we know surjectivity for rings of smaller length.
Pick $a \in A$ non-zero such that $a\mf_A=0$ and set $A' \coloneqq A/aA$, so that $\ell(A') = \ell(A)-1$.
Let $(\pi,\phi)$ be a lift of $\bar\pi$ over $A$ and set $\pi' \coloneqq \pi/a\pi$.
We set $\pi' \coloneqq \pi/a\pi$ which is free over $A'$.
Since $a\pi \subseteq \mf_A\pi$, $\phi$ factors through a surjection $\phi' : \pi' \twoheadrightarrow \bar\pi$ whose kernel is $\mf_{A'}\pi'$, so that $(\pi',\phi')$ is a lift of $\bar\pi$ over $A'$.
The multiplication by $a$ induces an isomorphism $\bar\pi \iso a\pi$, hence a short exact sequence of admissible $A$-representations of $G$
\begin{equation} \label{devpi}
0 \to \bar\pi \to \pi \to \pi' \to 0.
\end{equation}
By the induction hypothesis, there exists a lift $(\sigma',\psi')$ of $\bar\sigma$ over $A'$ and an isomorphism $\iota' : \St_{\bar Q}^G(\sigma') \iso \pi'$ such that $\St_{\bar Q}^G(\psi') = \phi' \circ \iota'$.
We deduce from Proposition \ref{prop:Ext} that there exists a short exact sequence of admissible $A$-representations of $L$
\begin{equation} \label{devsigma}
0 \to \bar\sigma \to \sigma \to \sigma' \to 0
\end{equation}
and an isomorphism $\iota : \St_{\bar Q}^G(\sigma) \iso \pi$ such that the diagram of admissible $A$-representations of $G$
\begin{equation} \label{deviso}
\begin{tikzcd}
0 \rar & \St_{\bar Q}^G(\bar\sigma) \dar[equal] \rar & \St_{\bar Q}^G(\sigma) \arrow[d,"\vertiso","\iota"'] \rar & \St_{\bar Q}^G(\sigma') \arrow[d,"\vertiso","\iota'"'] \rar & 0 \\
0 \rar & \bar\pi \rar & \pi \rar & \pi' \rar & 0
\end{tikzcd}
\end{equation}
where the upper row is obtained by applying $\St_{\bar Q}^G$ to \eqref{devsigma} and the lower row is \eqref{devpi}, is commutative.
In particular, $\St_{\bar Q}^G(\sigma)$ is free over $A$ so that $\sigma$ is also free over $A$.
Since $\mf_A\bar\sigma=0$ and $A \neq k$, the image of the second arrow of \eqref{devsigma} lies in $\mf_A\sigma$.
Thus, if we define $\psi : \sigma \twoheadrightarrow \bar\sigma$ to be the third arrow of \eqref{devsigma} composed with $\psi'$, then $\psi$ is surjective with kernel $\mf_A\sigma$, i.e.\@ $(\sigma,\psi)$ is a lift of $\bar\sigma$.
Moreover, there is a commutative diagram of admissible $A$-representations of $G$
\begin{equation*}
\begin{tikzcd}
\St_{\bar Q}^G(\sigma) \arrow[d,"\vertiso","\iota"'] \rar[two heads] & \St_{\bar Q}^G(\sigma') \arrow[d,"\vertiso","\iota'"'] \rar[two heads]{\psi'} & \St_{\bar Q}^G(\bar\sigma) \dar[equal] \\
\pi \rar[two heads] & \pi' \rar[two heads]{\phi'} & \bar\pi
\end{tikzcd}
\end{equation*}
where the left square is the right square of \eqref{deviso}.
By definition, the upper horizontal composite is $\psi$, and the lower composite is $\phi$, hence $\St_{\bar Q}^G(\psi) = \phi \circ \iota$.
\end{proof}

\section{Pro-representability}

Let $\Pro(\Oc)$ be the category of profinite local $\Oc$-algebras $A$ such that the structural morphism $\Oc \to A$ is local and induces an isomorphism $k \iso A/\mf_A$.
The morphisms $A \to A'$ are the continuous $\Oc$-algebra homomorphisms.
Note that $\Art(\Oc)$ is the full subcategory of $\Pro(\Oc)$ consisting of (discrete) artinian rings.
Moreover, $\Pro(\Oc)$ is equivalent to the category of pro-objects of $\Art(\Oc)$ (\cite[Lemma 3.3]{Schm}).

Let $\bar\pi$ be a smooth $k$-representation of $G$.
We say that $\Def_{\bar\pi}$ is pro-representable if there exists a universal deformation ring $R_{\bar\pi}^\univ \in \Pro(\Oc)$ so that there is a natural bijection
\begin{equation} \label{pro-rep}
\Hom_{\Pro(\Oc)}(R_{\bar\pi}^\univ,A) \iso \Def_{\bar\pi}(A)
\end{equation}
for all $A \in \Art(\Oc)$.
If $\End_G(\bar\pi)=k$, then $\Def_{\bar\pi}$ is pro-representable by \cite[Theorem 3.8]{Schm} whose proof works for any locally profinite group $G$.

\begin{coro} \label{coro:pro-rep}
Let $\bar\sigma$ be a smooth $k$-representation of $L$ trivial on $L \cap \GU$ and set $\bar\pi \coloneqq \St_{\bar Q}^G(\bar\sigma)$.
Assume that $\Def_{\bar\sigma}$ is pro-representable (e.g.\@ $\End_L(\bar\sigma)=k$).
If $\bar\sigma$ is admissible, then $\Def_{\bar\pi}$ is also pro-representable and there is an isomorphism $R_{\bar\pi}^\univ \simeq R_{\bar\sigma}^\univ$ in $\Pro(\Oc)$.
\end{coro}

\section{Noetherian deformations}

Let $\Noe(\Oc)$ be the category of noetherian complete local $\Oc$-algebras $A$ such that the structural morphism $\Oc \to A$ is local and induces an isomorphism $k \iso A/\mf_A$.
The morphisms $A \to A'$ are the local $\Oc$-algebra homomorphisms.
If $A \in \Pro(\Oc)$ is noetherian, then the profinite topology is the $\mf_A$-adic topology (\cite[Proposition 22.5]{pLG}).
Moreover, a morphism $A \to A'$ between noetherian rings in $\Pro(\Oc)$ is continuous if and only if it is local.
Thus $\Noe(\Oc)$ is the full subcategory of $\Pro(\Oc)$ consisting of noetherian rings.

A lift of a smooth $k$-representation $\bar\pi$ of $G$ over $A \in \Noe(\Oc)$ is a pair $(\pi,\phi)$ where $\pi$ is an $\mf_A$-adically continuous $A$-representation of $G$ orthonormalisable over $A$ (i.e.\@ the $A/\mf_A^n$-modules $\pi/\mf_A^n\pi$ are free for all $n \geq 1$) and $\phi : \pi \twoheadrightarrow \bar\pi$ is a surjection with kernel $\mf_A\pi$.
A morphism $\iota : (\pi,\phi) \to (\pi',\phi')$ is a morphism $\iota : \pi \to \pi'$ such that $\phi = \phi' \circ \iota$.
If $(\pi,\phi)$ is a lift of $\bar\pi$ over $A \in \Noe(\Oc)$ and $A \to A'$ is a morphism in $\Noe(\Oc)$, then $(\pi \otimesh_A A',\phi')$ is naturally a lift of $\bar\pi$ over $A'$, where the completed tensor product is defined by
\begin{equation*}
\pi \otimesh_A A' \coloneqq \varprojlim_{n \geq 1} \pi/\mf_A^n\pi \otimes_{A/\mf_A^n} A'/\mf_{A'}^n
\end{equation*}
and $\phi' : \pi \otimesh_A A' \twoheadrightarrow \bar\pi$ is the morphism induced by $\phi$.
Thus $\Def_{\bar\pi}$ extends to a functor $\Noe(\Oc) \to \Set$.

Let $\bar\sigma$ be a smooth $k$-representation of $L$ trivial on $L \cap \GU$ and set $\bar\pi \coloneqq \St_{\bar Q}^G(\bar\sigma)$.
Likewise $\Def_{\bar\sigma}$ extends to a functor $\Noe(\Oc) \to \Set$.
Using \eqref{StI}, we see that:
\begin{itemize}
\item $\St_{\bar Q}^G(\sigma)$ is orthonormalisable over $A$ if and only if $\sigma$ is,
\item $\St_{\bar Q}^G$ is compatible with base change: for any morphism $A \to A'$ in $\Noe(\Oc)$, there is a natural isomorphism $\St_{\bar Q}^G(\sigma) \otimesh_A A' \simeq \St_{\bar Q}^G(\sigma \otimesh_A A')$.
\end{itemize}
Thus \eqref{StDef} extends to a natural transformation between functors $\Noe(\Oc) \to \Set$.

\begin{prop} \label{prop:StIDef}
If $\bar\sigma$ is admissible, then \eqref{StDef} is an isomorphism of functors $\Noe(\Oc) \to \Set$.
\end{prop}

\begin{proof}
This is a formal consequence of Proposition \ref{prop:StDef} (see the proofs of \cite[Lemma 3.14 and Theorem 3.15]{HSS}).
As in the proof of Proposition \ref{prop:StDef}, we only need to assume that $\bar\sigma$ is locally $Z_M$-finite for the injectivity.
\end{proof}

Let $\bar\pi$ be a smooth $k$-representation of $G$ such that $\Def_{\bar\pi}$ is pro-representable.
By \cite[Corollary 3.9]{Schm}, which holds true for any locally profinite group $G$, $R_{\bar\pi} \in \Noe(\Oc)$ if and only if $\dim_k \Ext_G^1(\bar\pi,\bar\pi) < \infty$.
In this case, \eqref{pro-rep} extends to all $A \in \Noe(\Oc)$ and there exists a universal deformation $(\pi^\univ,\phi^\univ)$ of $\bar\pi$ over $R_{\bar\pi}^\univ$ so that \eqref{pro-rep} is induced by base change:
\begin{equation*}
(R_{\bar\pi}^\univ \to A) \mapsto (\pi^\univ \otimesh_{R_{\bar\pi}^\univ} A,\phi),
\end{equation*}
where $\phi : \pi^\univ \otimesh_{R_{\bar\pi}^\univ} A \twoheadrightarrow \bar\pi$ is the morphism induced by $\phi^\univ$.

\begin{coro} \label{coro:Noe}
Let $\bar\sigma$ be a smooth $k$-representation of $L$ trivial on $L \cap \GU$ and set $\bar\pi \coloneqq \St_{\bar Q}^G(\bar\sigma)$.
Assume that $\Def_{\bar\sigma}$ is pro-representable (e.g.\@ $\End_L(\bar\sigma)=k$) and $R_{\bar\sigma} \in \Noe(\Oc)$ (i.e.\@ $\dim_k \Ext_L^1(\bar\sigma,\bar\sigma) < \infty$).
If $\bar\sigma$ is admissible, then $\Def_{\bar\pi}$ is also pro-representable and there is an isomorphism $R_{\bar\pi}^\univ \simeq R_{\bar\sigma}^\univ$ in $\Noe(\Oc)$ via which $\pi^\univ=\St_{\bar Q}^G(\sigma^\univ)$.
\end{coro}

\section{The case of a character}

Let $\Lambda$ be the Iwasawa algebra over $\Oc$ of the pro-$p$ completion of the abelianisation of $G$, i.e.
\begin{equation*}
\Lambda \coloneqq \varprojlim_H \Oc[G/H]
\end{equation*}
where $H$ runs among the open normal subgroups of $G$ such that $G/H$ is an abelian $p$-group.
We let $\lambda : G \to \Lambda^\times$ denote the natural continuous group homomorphism.
By \cite[Proposition 19.7]{pLG}, we have $\Lambda \in \Pro(\Oc)$.
Moreover,
\begin{equation*}
\Lambda \in \Noe(\Oc) \Leftrightarrow \dim_{\F_p} \Hom_\Grp^\cont(G,\F_p) < \infty.
\end{equation*}
In particular, $\Lambda \in \Noe(\Oc)$ when $\car(F)=0$ by \cite[Proposition 3.12]{Schm}.

\begin{exem}
Assume $G=\GL_n(F)$ so that $\det$ induces an isomorphism of topological groups $G^\ab \iso F^\times$.
Using \cite[II Proposition 5.7]{Neu}, we deduce the following.
\begin{itemize}
\item If $\car(F)=0$, then $\Lambda \simeq \Oc[\mu_{p^\infty}(F)][\![X_1,\dots,X_d]\!] \in \Noe(\Oc)$ where $\mu_{p^\infty}(F)$ is the group of $p$-power roots of unity in $F$ and $d=[F:\Q_p]+1$.
\item If $\car(F)=p$, then $\Lambda \simeq \Oc[\![X_1,X_2,X_3,\dots]\!]$ is the algebra of formal power series in countably infinitely many indeterminates with coefficients in $\Oc$ (in the sense of \cite[IV §~4]{BbkA47-en}), which is not noetherian.
\end{itemize}
\end{exem}

Let $\bar\chi : G \to k^\times$ be a smooth character.
We let $\hat{\bar{\chi}} : G \to \Oc^\times$ denote the continuous character obtained by composing $\bar\chi$ with the canonical lifting $k^\times \hookrightarrow \Oc^\times$.
We define a continuous character $\chi^\univ : G \to \Lambda^\times$ by setting
\begin{equation*}
\chi^\univ(g) \coloneqq \hat{\bar{\chi}}(g) \lambda(g)
\end{equation*}
for all $g \in G$.
Proceeding as in the proof of \cite[Proposition 3.11]{Schm} (here $G^\ab$ need not be topologically finitely generated because $H$ is required to be open in the definition of $\Lambda$), we see that $\Lambda$ is the universal deformation ring of $\bar\chi$ and $\chi^\univ$ is the universal deformation of $\bar\chi$.

\begin{coro} \label{coro:Lambda}
The universal deformation ring of $\St_{\bar Q}^G(\bar\chi)$ is $\Lambda$.
If $\Lambda \in \Noe(\Oc)$ (e.g.\@ if $\car(F)=0$), then the universal deformation of $\St_{\bar Q}^G(\bar\chi)$ is $\St_{\bar Q}^G(\chi^\univ)$.
\end{coro}

\bibliographystyle{amsalpha}
\bibliography{steinberg}

\end{document}